\documentclass[12pt]{article}
\usepackage[margin=1.5cm]{geometry}

\usepackage{longtable}
\usepackage{latexsym,amsmath,amsfonts,amsthm,color,mathtools,adjustbox}

\newtheorem{theorem}{Theorem}[section]
\newtheorem{prop}[theorem]{Proposition}
\newtheorem{lemma}[theorem]{Lemma}
\newtheorem{problem}[theorem]{Problem}
\newtheorem{cor}[theorem]{Corollary}
\newtheorem*{definition*}{Definition}

\newcommand{\rank}{\mathop{\mathrm{rank}}\nolimits}

\newcommand{\Sp}{\mathop{\mathrm{Sp}}}

\newcommand{\trans}{\mathop{\mathrm{T}}}
\newcommand{\ptrans}{\mathop{\mathrm{PT}}}
\newcommand{\dom}{\mathop{\mathrm{dom}}}
\newcommand{\im}{\mathop{\mathrm{im}}}

\newcommand{\sym}{S_n}

\newcommand{\End}{\mathop{\mathrm{End}}}
\newcommand{\aut}{\mathop{\mathrm{Aut}}}

\begin{document}
\title{A Transversal Property for Permutation Groups Motivated by Partial Transformations}
\author{Jo\~ao Ara\'ujo\footnote{Center for Mathematics and Applications (CMA) and Departamento de Matem\'atica, Faculdade de Ci\^encias e Tecnologia (FCT)
Universidade Nova de Lisboa (UNL), 2829-516 Caparica, Portugal;   jj.araujo@fct.unl.pt} 
 \footnote{CEMAT-Ci\^{e}ncias,
Faculdade de Ci\^{e}ncias, Universidade de Lisboa
1749--016, Lisboa, Portugal; jjaraujo@fc.ul.pt} ,
Jo\~ao Pedro Ara\'ujo\footnote{IST, University of Lisbon},
Wolfram Bentz\footnote{Department of Physics and Mathematics, University of Hull, Kingston upon Hull, HU6 7RX, UK; W.Bentz@hull.ac.uk},\\
Peter J. Cameron\footnote{School of Mathematics and Statistics, University of St Andrews, North Haugh, St Andrews, Fife KY16 9SS, UK; pjc20@st-andrews.ac.uk},
and Pablo Spiga\footnote{Dipartimento di Matematica e Applicazioni, University
of Milano-Bicocca, Milano, 20125 Via Cozzi 55, Italy; pablo.spiga@unimib.it}}
\date{}
\maketitle
\begin{abstract}
In this paper we introduce the definition of $(k,l)$-universal transversal property, which is a refinement of the definition of $k$-universal transversal property, which in turn is a refinement of the classic definition of $k$-homogeneity for permutation groups. In particular, a group possesses the $(2,n)$-universal transversal property if and only if it is primitive; it possesses the $(2,2)$-universal transversal property if and only if it is $2$-homogeneous. Up to a few undecided cases, we give a classification of groups satisfying the $(k,l)$-universal transversal property, for $k\ge 3$. Then we apply this result for studying regular semigroups of partial transformations.
\end{abstract}
\section{Introduction} 

Let $G$ be a permutation group on $\Omega$ of degree~$n$. 
For $k\in \mathbb{N}$, the group $G$ is $k$-{\em homogeneous} if it acts transitively on the
set of $k$-subsets of $\Omega$. Also, we say that $G$ possesses the
{\em $k$-universal transversal property} (or $k$-ut property for short) if the
orbit of any $k$-subset contains a transversal for any $k$-partition of $\Omega$
(a property studied in~\cite{ac}). It is clear that $k$-homogeneity implies
the $k$-ut property.

One of our main goals in this paper is to refine this observation, in the following way.
\begin{definition*}{\rm 
Let  $k$ and $l$ be integers with $1\le k\le l\le n$ and let $G$ be a  permutation
group on $\Omega$  of degree $n$. Then $G$ is said to have the \emph{$(k,l)$-universal transversal property}
(or $(k,l)$-ut for short) if, given any $k$-subset $A$ and any $l$-subset $B$ of
$\Omega$ and a $k$-partition $P$ of $B$, there exists $g\in G$ such that $Ag$
is a transversal for $P$.}
\end{definition*} Now it is clear that the $(k,k)$-ut property is just
$k$-homogeneity and the $(k,n)$-ut property is the $k$-ut property. Moreover, when $k=n$, every group satisfies the $(k,l)$-ut property and hence in what follows we will always assume $k<n$.

We show in Proposition~\ref{up}~\eqref{wards} that, for groups with the $k$-ut property, $(k,l)$-ut becomes
weaker as $l$ increases; so, for any group $G$ satisfying $k$-ut, there is a
threshold $t(G,k)$ such that $G$ satisfies $(k,l)$-ut if and only if
$n\ge l\ge t(G,k)$. 

The aim of this paper is to investigate the $(k,l)$-ut property and its impact on
semigroups of partial transformations. The group-theoretic part of the investigation falls into two quite
separate parts. For $k=1$ and $k=2$, there is no hope of a general
classification: every permutation group has the $1$-ut property, whereas
$2$-ut is equivalent to primitivity. Moreover, the threshold $t(G,k)$ for
$k\in \{1,2\}$ is simply expressed in terms of well-understood parameters
(orbit lengths when $k=1$, and valencies of orbital graphs when $k=2$), and can
be efficiently computed for any permutation group (see Propositions~\ref{k=1}
and~\ref{k=2}).

On the other hand, the results of \cite{abc-et,ac} give a nearly complete
classification of groups with the $k$-ut property for $3\le k\le n-2$, and
so it is entirely reasonable to ask for a similar classification of groups
satisfying $(k,l)$-ut for any such $k$ and all $l$ with $k\le l\le n$; and
this we do (up to a few unresolved cases). The results are summarised below.

\begin{theorem}\label{upperhalf}
Let $G$ be a permutation group of degree $n$ having the $k$-ut property with $n>k > \lfloor (n+1)/2\rfloor$. Then $G$ is $k$-homogeneous, and $t(G,k)=k$. 
\end{theorem}

\begin{theorem}\label{k>5}
Let $G$ be a permutation group of degree $n$ having the $k$-ut property with $6\le k\le \max\{\lfloor (n+1)/2\rfloor,n-6\}$. Then $G$ is either the alternating or the symmetric group of degree $n$. In particular, $t(G,k)=k$.
\end{theorem}

\begin{theorem}\label{k=5}
Let $G$ be a permutation group of degree $n$ having the $5$-ut property. % with $n\ge 9$. 
 Then either $t(G,5)=5$ and hence $G$ is $5$-homogeneous, or $n=33$, $G=\mathrm{P}\Gamma\mathrm{L}(2,32)$ and $t(G,5)=30$.
\end{theorem}

\begin{theorem}\label{k=4}
Let $G$ be a permutation group of degree $n$ having the $4$-ut property.  
 Then either $t(G,4)=4$ and hence $G$ is $4$-homogeneous, or one of the following holds:
\begin{enumerate}
\item\label{case41} $n=12$, $G=M_{11}$ and $t(G,4)=10$, 
\item\label{case42} $n=8$, $G=\mathrm{PGL}(2,7)$ and $t(G,4)=7$, 
\item\label{case42b} $n=7$, $G=\mathrm{AGL}(1,7)$ and $t(G,4)=7$,
\item\label{case43} $n=q+1$, $\mathrm{PSL}(2,q) \le G \le \mathrm{P}\Gamma\mathrm{L}(2,q)$, $q\equiv 11\pmod {12}$ or $q=2^p$ with $p$ prime, and for all $c\in\mathrm{GF}(q)\setminus\{0,1\}$ the elements $-1$, $c$ and $c-1$ generate the multiplicative group of the field $\mathrm{GF}(q)$.
\end{enumerate}
\end{theorem}
We are not sure if case~\eqref{case43} actually arises for $n\ge 50$ and what  the value for $t(G,4)$ is in case it does arise.

\begin{theorem}\label{k=3}
Let $G$ be a permutation group of degree $n$ having the $3$-ut property. %with $n\ge 5$. 
Then either $t(G,3)=3$ and hence $G$ is $3$-homogeneous, or $(G,n,t(G,3))$ is in Table~$\ref{table:1}$.
\end{theorem}

\begin{table}[htbp]
\begin{center}
\begin{adjustbox}{angle=90}
\begin{tabular}{|c|c|c|c|}\hline
$n$&$G$&$t(G,3)$& Comments\\\hline
$q+1$& $\mathrm{PSL}(2,q) \le G \le \mathrm{P\Sigma L}(2,q)$&$\frac{q+5}{2}\le t(G,3)\le \frac{3q+11}{5}$&  $q = 1 \pmod 4$\\
$2^{2d-1} + 2^{d-1}$& $\Sp(2d,2)$&$2^{2d-2}+2^{d-1}+1\le t(G,3)\le \frac{3\cdot 2^{2d-1}+3 \cdot 2^{d}+2}{5}$ &$d \ge 3$\\
$28$ & $\Sp(6,2)$ & $19$ &\\
$2^{2d-1} - 2^{d-1}$& $\Sp(2d,2)$&$2^{2d-2}+3\le t(G,3)\le \frac{3\cdot 2^{2d-1}+14}{5}$ &$d \ge 4$\\
$2^{2d}$& $2^{2d} \colon \Sp(2d,2)$&$2^{2d-1}+3\le t(G,3)\le \frac{3\cdot 2^{2d}+14}{5}$ &$d\ge 2$\\
$5$& $\mathrm{C}_5$ & $5$ & \\
$5$& $\mathrm{D}(2*5)$ & $5$ & \\
$7$& $\mathrm{AGL}(1,7)$&$6$ & \\
$11$& $\mathrm{PSL}(2,11)$&$9$ & \\
$16$& $2^6:A_6$&$11$ &\\
$64$& $2^6:\mathrm{G}_2(2)'$&$59\le t(G,3)\le 64$ &\\
$64$& $2^6:\mathrm{G}_2(2)$&$59\le t(G,3)\le 64$ &\\
$65$& $\mathrm{Sz}(8)$&$45\le t(G,3)\le 65$ &\\
$65$& $\mathrm{Aut}(\mathrm{Sz}(8))$&$45\le t(G,3)\le 65$ &\\
$176$& $HS$&$165\le t(G,3)\le 176$ &\\
$276$& $Co_3$&$165\le t(G,3)\le 169$ &\\
$q^2+1$&$\mathrm{Sz}(q)\unlhd G\le \mathrm{Aut}(\mathrm{Sz}(q))$&-&no information on $t(G,3)$, $q=2^{2d+1}, d\ge 2$\\
$q$&$G\le \mathrm{A}\Gamma\mathrm{L}(1,q)$&-&no information on $t(G,3)$\\
&&&$\mathrm{AGL}(1,q)\le G$ or $|\mathrm{AGL}(1,q):G|=2$\\ 
&&& $q$ is prime with $q \equiv 11 \mbox{ mod } 12$, or\\
&&& $q=2^p$ with $p$ prime, and\\
&&&   $\forall c \in \mathrm{GF}(q)\setminus\{0,1\}$, 
$|\langle -1,c,c-1\rangle|=q-1$\\\hline
\end{tabular}
\end{adjustbox}
\end{center}
\label{table:1}
\end{table}

%In the light of Theorems~\ref{k=4} and~\ref{k=3}, we propose the following problem.
%\begin{problem}{\rm Improve upon the values of $t(G,3)$ in Table~\ref{table:1} and determine whether case~\eqref{case43} in Theorem~\ref{k=4} arises in general.
%}
%\end{problem}

This machinery is used to obtain a result about transformation semigroups. Before stating the main result we introduce some notation. Let $\Omega$ be a set; a {\em partial transformation} $t$ on $\Omega$ is a function $t:\Delta\to \Omega$, where $\Delta$ is an arbitrary subset of $\Omega$. The set of all such partial transformations forms a semigroup $\ptrans(\Omega)$ under partial composition. Let $\ptrans_{k,l}(\Omega)$ denote the set of all partial transformations $t:\Delta\to \Omega$ on $\Omega$  such that $|\Omega t|=k$ and $|\Delta|=l$. Recall that a semigroup $S$ is {\em regular} if, for all $a\in S$, there exists $b\in S$ such that $a=aba$.
\begin{theorem}\label{thrm:main}
Let $G$ be a permutation group on $\Omega$ of degree $n$ and consider the following three properties:
\begin{enumerate} 
\item\label{c)} $G$ possesses the $(k,l)$-ut property. 
\item\label{a)} For all $t\in \ptrans_{k,l}(\Omega)$, the semigroup $\langle G,t\rangle$ is regular.
\item\label{b)} For all $m\ge l$ and all $t\in \ptrans_{k,m}(\Omega)$, the semigroup $\langle G,t\rangle$ is regular.
\end{enumerate} 
Then, ${(c)}$ implies $(b)$ and $(b)$ implies $(a)$; moreover,  when $k\le \max(\lfloor(n+1)/2\rfloor,n-6)$, $(a)$ implies $(c)$ and hence $(a)$, $(b)$ and $(c)$ are equivalent.
\end{theorem}

In the past, the {\em praxis} was to consider a semigroup theory problem solved whenever it was reduced to a problem in group theory.  This dramatically changed in recent years as it turned out to be much more instructive  for both sides to keep an ongoing conversation. One of the driving forces of this conversation has been the following general problem:
\begin{quote}
Classify the pairs $(G,a)$, where $a$ is a map on a finite set $\Omega$ and $G$ is a group of permutations of $\Omega$,
such that the semigroup  $\langle G,a\rangle$ generated by $a$ and $G$  has a given property $P$.
\end{quote}

A very important class of groups that falls under this general scheme is that of \emph{synchronizing groups}, groups of permutations on a set that together with any non-invertible map on the same set generate a constant (see \cite{steinberg,abc,abcrs,arcameron22,acs,cameron,neumann}). These groups are very interesting from a group theoretic  point of view and are linked  to the \emph{\v{C}ern\'y conjecture},  a longstanding open problem in automata theory.

Another instance of the general question is the following.
Let $A\subseteq \trans (\Omega)$, the full trasnformation monoid on $\Omega$; classify the permutation groups $G\le \sym$ such that $\langle G,a\rangle$ is regular for all $a\in A$. For many different sets $A$, this problem has been considered in    \cite{abc2,abc-et,ac,AMS,lmm,lm,levi96,mcalister}, among others. The goal of this paper is to consider the similar problem when $A$ is a set of partial transformations with prescribed domain and image sizes (Theorem \ref{thrm:main}).

Section \ref{scet2} contains some general results. In Sections~\ref{case1} and~\ref{case2} we deal
with the cases $k=1$ and $k=2$. Sections \ref{case3} and \ref{case3b} deal with the cases $k\ge 3$. In Section \ref{semigroups} we use the previous results to extract our main application on semigroups. 

\section{General results}\label{scet2}

We begin by collecting some observations about the $(k,l)$-ut  property.

\begin{prop} \label{up} Let $G$ be a permutation group of degree $n$. 
\begin{enumerate}
\item\label{warda} The $(k,k)$-ut  property is equivalent to $k$-homogeneity and the $(k,n)$-ut  property is equivalent to the $k$-ut property.
\item\label{wardb} The $(2,n)$-ut property is equivalent to primitivity.
\item\label{wards} For fixed $k$, if $G$ possesses $(k,l)$-ut, then it possesses
$(k,m)$-ut, for all $l\le m\le n$. 
\end{enumerate}
\end{prop}

\begin{proof}
Part~\eqref{warda} is clear and part~\eqref{wardb} follows from~\cite[Theorem~$1.8$]{ac}.

For part~\eqref{wards}, suppose that $G$ has the
$(k,l)$-ut property, with $l< n$. We show that $G$ has the $(k,l+1)$-ut property. To this end,  let $\Omega$ be the domain of $G$, let $A\subseteq \Omega$ be a $k$-subset, let $B \subseteq \Omega$ be an $(l+1)$-subset, and let $P$ be a
$k$-partition of $B$. Let $b$ be an element of $B$ lying in a part of $P$ with more than one element: observe that this is possible from the pigeonhole principle  because $l+1>l\ge k$.  Now, let $\bar{B}$ be the set $B\setminus \{b\}$ and let $\bar{P}$ be the partition obtained from $P$ by removing $b$ from the part containing $b$ and leaving all other parts unchanged. Since $G$ has the $(k,l)$-ut property, some $G$-conjugate of $A$ is a section for $\bar{P}$ and hence also a section for $P$. Applying this repeatedly
proves part~\eqref{wards}.
\end{proof}

As already noted, Proposition~\ref{up}~\eqref{wards} implies that, for a permutation group $G$
satisfying  the $k$-ut property, there is a threshold $t(G,k)\ge k$
such that $G$ satisfies $(k,l)$-ut  if and only if $n \ge l\ge t(G,k)$. Much of our
investigation will concern the value of this threshold.

\begin{prop}\label{up2}
Let $2\le k\le l$, let $G$ be a permutation group on $\Omega$ of degree $n$ having the 
property that the setwise stabiliser of any $k$-subset acts transitively on it, and let $a\in\Omega$. 
\begin{enumerate}
\item\label{partb}
If $G$ has the $(k,l)$-ut property, then the point stabiliser $G_a$ in its action on $\Omega\setminus\{a\}$ has the $(k-1,l-1)$-ut property.
\item\label{partc}The group $G$ is $(k-1)$-homogeneous.
\end{enumerate}
In particular, if $G$ has the $k$-ut property, then $t(G,k)\ge t(G_a,k-1)+1$.

\label{p:induct}
\end{prop}

\paragraph{Remark}
This is not the first time that the hypothesis in this
proposition has occurred in the literature and, for instance, it is relevant for investigating min-wise independent sets of permutations, see~\cite[Corollary~$1$]{CS}. In particular, groups satisfying this hypothesis are classified in~\cite[Theorem~$2$]{CS}.

\begin{proof}
We start by proving~\eqref{partb}: suppose that $G$ is a permutation group of degree $n$ having the property that the setwise stabiliser of any $k$-subset acts transitively on it and suppose that $G$ has the $(k,l)$-ut property. 

Let $A$ be a $(k-1)$-subset of $\Omega\setminus\{a\}$, let $B$ be
 an $(l-1)$-subset of $\Omega\setminus\{a\}$,  and let
$P$ be a $(k-1)$-partition of $B$. 
By assumption, the $k$-subset
$\{a\}\cup A$ can be mapped to a transversal of the $k$-partition
$\{\{a\}\}\cup P$ of $\{a\}\cup B$ by an element $g\in G$. By transitivity, we
can premultiply $g$ by an element of the setwise stabiliser of $\{a\}\cup A$
to ensure that $a$ maps into the part $\{a\}$ of the
partition. The resulting element of $G$ belongs to the stabiliser of $a$
and maps $A$ to a transversal for $P$. Thus part~\eqref{partb} is proven.

Part~\eqref{partb} immediately implies $t(G,k)\ge t(G_a,k-1)+1$ for groups satisfying the $k$-ut property. For part~\eqref{partc}, observe that the hypothesis in this proposition implies that we can move, via elements in $G$, a $(k-1)$-subset to any $(k-1)$-subset which
meets it in $k-2$ points, and the general case follows by induction.
\end{proof}

\medskip

\paragraph{Remark}From~\cite[Theorem~$3.3$]{ac}, we see that for $n>k>\lfloor (n+1)/2\rfloor$, $k$-ut is equivalent to
$k$-homogeneity. In particular, for a fixed $k$ with $n>k>\lfloor (n+1)/2\rfloor$, all the $(k,l)$-ut  properties for $k\le l\le n$ are equivalent by Proposition~\ref{up}. Therefore, for our investigation we need to consider
only $k\le\lfloor(n+1)/2\rfloor$. However, we will occasionally replace this bound with $k\le \max\{\lfloor (n+1)/2\rfloor,n-6\}$ to facilitate our semigroup application.

Observe that when $k=n$, every permutation group satisfies the $k$-ut property.

\medskip

We now note the implications of an important result from \cite{ac} for our
problem.

\begin{prop}\label{Living}
Let $G$ be a permutation group of degree $n$. Suppose that $G$ has the $(k,l)$-ut property with $2\le k\le \max\{\lfloor (n+1)/2\rfloor,n-6\}$.
Then either $G$ has the $(k-1,m)$-ut property for every $m$ with $k-1\le m\le n$ or one of the following holds:
\begin{enumerate}
\item $n=5$, $k=3$, $t(G,k)=5$ and $G$ is cyclic or dihedral,
\item $n=7$, $k=4$, $t(G,k)=7$ and $G$ is isomorphic to $\mathrm{AGL}(1,7)$.
\end{enumerate}
\end{prop}

\begin{proof}
If $G$ has $(k,l)$-ut, then it has $(k,n)$-ut from Proposition~\ref{up}~\eqref{wards}. Therefore $G$ has $k$-ut from Proposition~\ref{up}~\eqref{warda}.
 By \cite[Proposition 4.3(2)]{ac}, we deduce that one of the following holds:
\begin{itemize}
\item $G$ is $(k-1)$-homogeneous,
\item $n=5$, $k=3$ and $G$ is cyclic or dihedral,
\item $n=7$, $k=4$ and $G$ is isomorphic to $\mathrm{AGL}(1,7)$,
\item $k>\lfloor (n+1)/2\rfloor$.
\end{itemize}
(Observe that~\cite[Proposition 4.3(2)]{ac} lists five exceptions; however, by a computation one can easily check that only three actually arise.) 

If $G$ is $(k-1)$-homogeneous, then $G$ satisfies $(k-1,k-1)$-ut from Proposition~\ref{up}~\eqref{warda}. Therefore, by Proposition~\ref{up}~\eqref{wards}, $G$
satisfies $(k-1,m)$-ut for any $m$ with $k-1\le m\le n$.

In the second and in the third case, a computation yields $t(G,k)=n$ as required, and we obtain the two exceptions listed in the statement of this lemma.

Finally, suppose $k>\lfloor (n+1)/2\rfloor$. As $k\le \max\{(n+1)/2,n-6\}$, we have $k\le n-6$. From the remark preceding Proposition~\ref{Living}, we deduce $G$ is $k$-homogeneous and hence $(n-k)$-homogeneous. As $n-k\ge 6$, from the Livingstone-Wagner theorem, we get that $G$ is $6$-transitive and hence $G$ is the symmetric or the alternating group of degree $n$. In particular, the lemma follows immediately.
\end{proof}

Proposition~\ref{Living} is a natural analogue of the Livingstone--Wagner theorem~\cite{lw} and it will be crucial
for our application to semigroups.

%{\bf For this to work in the semigroups part we need to observe that $(k,l)$-ut implies $(k,n)$-ut implies $k$-ut implies $(k-1)$-homogeneous. I have there in bold the spot where it is needed. }

\section{The case $k=1$}\label{case1}

\begin{prop}
Let $G$ be a permutation group of degree $n$, let $1\le l\le n$ and let $d$ be the size of a smallest
$G$-orbit. Then $G$
possesses the $(1,l)$-ut property if and only if every orbit has size at least
$n-l+1$. Equivalently  $t(G,1)=n-d+1$. 
\label{k=1}
\end{prop}

\begin{proof}
The group $G$ has the $(1,l)$-ut property if and only if every $G$-orbit intersects
every $l$-subset, which is true if and only if every $G$-orbit has size at least
$n-l+1$.
\end{proof}

\begin{cor}Let $G$ be a permutation group of degree $n$. If $G$ is transitive, then $t(G,1)=1$; if $G$ is intransitive, then  $t(G,1)\ge n/2+1$. 
\end{cor}

\begin{proof}
If $G$ is transitive, then $G$ has $(1,m)$-ut, for every $1\le m\le n$, and hence $t(G,1)=1$. If $G$ is intransitive, then the smallest $G$-orbit has size at most $n/2$ and hence by Proposition~\ref{k=1} $t(G,1)\ge n-n/2+1=n/2+1$.
\end{proof}

So the threshold $t(G,1)$ cannot lie between $2$ and $(n+1)/2$. We will see that
this is an instance of a general phenomenon.

\section{The case $k=2$}\label{case2}

Recall, for instance from~\cite[Theorem~$1.8$]{ac2}, that the following conditions are equivalent: 
\begin{itemize} 
\item $G$ is primitive; 
\item every
orbital graph for $G$ (graph whose edges form a $G$-orbit on $2$-subsets) is
connected; and 
\item $G$ has the $2$-ut property.
\end{itemize}
We remark that this result uses the convention that the trivial group acting on two elements is primitive. 

\begin{prop}\label{k=2}
Let $G$ be a primitive permutation group of degree $n$, let $l\in\mathbb{N}$ and let $d$ be the
smallest valency of an orbital graph for $G$. Then $G$ possesses the $(2,l)$-ut property if
and only if $l\ge n-d+1$. Equivalently, $t(G,2)=n-d+1$.
\end{prop}

\begin{proof}Let $\Omega$ be the domain of $G$.
We observe first that $G$ has $(2,l)$-ut if and only if the induced subgraph
of any orbital graph on any $l$-subset of $\Omega$ is connected. For the
requirement is that each orbital graph has an edge which is a transversal for
every $2$-partition of every $l$-subset.

Suppose that $l\le n-d$. Let $\Gamma$ be an orbital graph of valency $d$ and let $B$ be an $l$-subset consisting of a vertex $v$ and $l-1$ non-neighbours of
$v$. Then the induced subgraph of $\Gamma$ on $B$ is disconnected.

Suppose that $l>n-d$. Let $\Gamma$ be an orbital graph for $G$ and let $B$ be an $l$-subset of $\Omega$. 
Now, we use the result of Watkins~\cite{watkins} asserting that
the vertex-connectivity of a vertex-primitive graph is equal to its valency.
(Watkins does not state this but it is implicit in his proof.) In our
case, every orbital graph has connectivity at least $d$, meaning that the
deletion of $d-1$ or fewer vertices leaves a connected graph. Since  $|\Omega\setminus B|=n-l\le d-1$, we deduce that the subgraph induced by $\Gamma$ on $B$ is connected.
\end{proof}

\begin{cor}\label{cor:=k=2}
Let $G$ be a primitive permutation group of degree $n$.
If $G$  is $2$-homogeneous, then $t(G,2)=2$; if $G$ is not $2$-homogeneous, then $t(G,2)\ge (n+3)/2$.
\end{cor}

\begin{proof}
If $G$ is not $2$-homogeneous, then there are at least two orbital
graphs, and so the smallest valency of an orbital graph is at most $(n-1)/2$. Hence, by Proposition~\ref{k=2}, $t(G,2)\ge n-(n-1)/2+1=(n+3)/2$.
\end{proof}

In particular, the threshold $t(G,2)$ cannot lie between $3$ and $n/2+1$. We can actually slightly improve Corollary~\ref{cor:=k=2}.

\begin{cor}\label{c:k=2}
Let $G$ be a primitive permutation group of degree $n>2$. Then
\begin{enumerate}
\item $G$ has $(2,n-1)$-ut.
\item\label{partb1} Either $G$ has $(2,n-2)$-ut, or $n$ is prime and $G$ is the cyclic or
the dihedral group of degree $n$.
\end{enumerate}
\end{cor}

\begin{proof}
Since orbital graphs are connected, they have valency at least~$2$ when $n>2$. Moreover, if an
orbital graph does have valency $2$, then it is a cycle, and $G$ is  as in (\ref{partb1}). The results now follows from Proposition~\ref{k=2}.
\end{proof}

\paragraph{Remark} Corollary~\ref{c:k=2}
can be pushed a little further. The primitive groups with an orbital
graph of valency $3$ were determined by Wong~\cite{wong}, following the
pioneering work of Sims~\cite{sims} on the connection between permutation
groups and graphs. This was extended to valency $4$ by Wang~\cite{wang},
using preliminary results of Sims and Quirin, and to valency $5$ by
Fawcett \textit{et~al.}~\cite{fglprv}. So, with the exception of the groups
in these classifications, a primitive group has $(2,n-5)$-ut.

However, the proofs increase in length and complexity as the valency increases,
and it is perhaps unlikely that this sequence of results will be continued in the new future.

\medskip

At the other end of the scale, if $G$ is not $2$-homogeneous and has $(2,(n+3)/2)$-ut, then there
must be two orbital graphs each of valency $(n-1)/2$. There are three
possibilities:
\begin{enumerate}
\item the stabiliser of a point has two self-paired suborbits of size $(n-1)/2$;
\item the stabiliser of a point has a self-paired suborbit of size $(n-1)/2$, and two paired
suborbits of size $(n-1)/4$;
\item the stabiliser of a point has four suborbits of size $(n-1)/4$, falling into two pairs of paired suborbits.
\end{enumerate}

In the third case, $G$ has odd order (since an involution in $G$ would imply
the existence of a self-paired suborbit), and so is soluble. Therefore, in the third case, $G$ is an
affine group, that is, $G$ is permutation isomorphic to a subgroup of the affine general linear group $\mathrm{AGL}(V)$, where $V$ is a finite vector space over a finite field of odd order. Then, the group generated by $G$ and $-I_V$ falls under case (a).
In case (a), we can invoke the classification of rank~$3$ primitive groups. The finite primitive groups of rank $3$ of affine type (according to the O'Nan-Scott partition into types) were classified by Liebeck~\cite{LiebeckAffine}. Kantor and Liebler~\cite{KantorLiebler} have classified the rank $3$ primitive permutation groups having socle a finite classical group, and Liebeck and Saxl~\cite{LiebeckSaxl} have classified the rank $3$ primitive permutation groups having socle a finite exceptional group of Lie type or a sporadic simple group. The papers~\cite{KantorLiebler,LiebeckAffine,LiebeckSaxl} contain all the relevant information for the complete classification of rank $3$ primitive permutation groups.
For dealing with Case (a), one might invoke also the classification of $3/2$-transitive groups by
Bamberg \emph{et al.}~\cite{bglps}. The conclusion (for case~(a) and (c)) is that either $G$ is an
affine group, or $G$ is the alternating or the symmetric group of degree $7$ acting on $2$-subsets (with degree $21$).

Case (b) is a little more problematic. In this case $G$ has rank~$4$, and
(arguing as before) we can assume that $G$ is not affine. Except for the primitive
groups of affine type, Cuypers in his D. Phil. thesis~\cite{cuypers} has classified the primitive groups having rank 4. 

There is some recent interest in pushing these classifications a little further for answering a question of Muzychuk concerning coherent configurations, see~\cite{MS}. However, since there are various primitive rank~$3$ permutation groups with nearly
equal subdegrees, we cannot improve the bound any further by using these investigations.

\section{The cases $k\ge3$: preliminaries}\label{case3}

A {\em hypergraph} is an ordered pair $(\Omega,E)$, where $E$ is a subset of the power set of $\Omega$. We refer to the elements of $E$ as the {\em edges} of the hypergraph. A hypergraph is $k$-{\em uniform} if each edge is a $k$-subset. Moreover, we say that a $k$-uniform hypergraph is \emph{regular} if any $(k-1)$-subset is
contained in a constant number of edges, and call this number the
\emph{valency} of the hypergraph. (This terminology is not standard.)

\begin{prop}\label{hyp}
Let $G$ be a permutation group of degree $n$ with the $k$-ut
property and $k\le\max\{(n+1)/2,n-6\}$, let $k\le l\le n$ and let $d_k$ be the smallest valency of a regular
$G$-invariant $k$-uniform hypergraph. If $G$ has the $(k,l)$-ut property,
then $l\ge n-d_k+1$.
\end{prop}

\begin{proof}
If $k\in \{1,2\}$, then the result follows from Propositions~\ref{k=1} and~\ref{k=2}. Suppose then $k\ge 3$. 

Assume that $G$ is
$(k-1)$-homogeneous. Then, every $G$-orbit on $k$-subsets forms a regular hypergraph.
Suppose that
$k\le l\le n-d_k$. Let $A$ be a $k$-subset which is an edge in a
hypergraph $E$
of smallest valency, and let $K$ be any $(k-1)$-subset. Let
\[B=K\cup\{x:K\cup\{x\}\notin E\},\]
so that $B=n-d_k$. Take the partition $P$ of $B$ consisting of singleton parts
containing the points of $K$ with all the rest in a single part. If $Ag$ 
is a
transversal for $P$, then $K\subseteq Ag$, but by 
construction, the
remaining point of $Ag$ lies outside $B$, a contradiction. So $G$ does not have the
$(k,n-d_k)$-ut property, whence $t(G,k)\ge n-d_k+1$.

Assume that $G$ is not $(k-1)$-homogeneous. From Proposition~\ref{Living}, we have only three exceptional cases to consider and in all cases $t(G,k)=n$. Hence the bound $t(G,k)\ge n-d_k+1$ is trivially satisfied.
 \end{proof}

Despite the fact that in Propositions~\ref{k=1} and~\ref{k=2} we obtained an explicit value $n-d_k+1$ for $t(G,k)$ when $k\in \{1,2\}$, we cannot hope to have this equality for $k\ge3$. For instance, $\mathrm{AGL}(1,7)$ has the $3$-ut property, with $t(\mathrm{AGL}(1,7),3)=7$, but the smallest valency of a regular $\mathrm{AGL}(1,7)$-invariant $3$-uniform hypergraph is $2$.

\begin{cor}
Let $G$ be a permutation group of degree $n$ having the $k$-ut property with $k\le \max\{\lfloor (n+1)/2\rfloor,n-6\}$. If $G$ is $k$-homogeneous, then $t(G, k) = k$; if $G$ is not $k$-homogeneous, then $t(G, k) \ge (n +
k+1)/2$.
\end{cor}

\begin{proof}
If $G$ is not $k$-homogeneous, then it has at least two orbits on $k$-subsets,
and the smallest regular $G$-invariant $k$-uniform hypergraph has valency at most $(n-(k-1))/2$. The proof now follows from Proposition~\ref{hyp}.
\end{proof}

\section{The cases $k\ge 3$: details}\label{case3b}
We have nearly complete classifications of finite groups satisfying the $k$-ut property for $k\ge 3$. We can hence go through this classification to determine the threshold value $t(G,k)$.

\subsection{Case $k=3$}
The groups satisfying $3$-ut were partially classified in \cite{ac}; in \cite{abc-et}, these results were extended and  some minor corrections added to obtain the following classification.

\begin{prop}\label{p:3ut}  Every group in the following list has the $3$-ut property.
\begin{enumerate}
\item\label{part1} $G$ is $3$-homogeneous;
\item\label{part2}$n=q+1$ and $\mathrm{PSL}(2,q)\le G \le \mathrm{P}\Sigma \mathrm{L}(2,q)$, where $q \equiv 1 \pmod 4$; 
\item\label{part3}$n=2^{2d-1}\pm 2^{d-1}$, $G=\mathrm{Sp}(2d,2)$, $d \ge 3$, in either of its $2$-transitive representations;
\item\label{part4}$n=2^{2d}$, $G=2^{2d}: \mathrm{Sp}(2d,2)$, $d \ge2$;
\item\label{part5} $(n,G)$ is one of $(5,\mathrm{C}_5)$, $(5,\mathrm{D}(2*5))$, $(7,\mathrm{AGL}(1,7))$, $(11,\mathrm{PSL}(2,11))$, $(16,2^4:A_6)$, $(64,2^6:G_2(2))$, $(64,2^6:G_2(2)')$, $(65,\mathrm{Sz}(8))$, $(65,\mathrm{Sz}(8):3)$, $(176,HS)$, $(276,\mathrm{Co}_3)$.
\end{enumerate}

If there is any other group with $3$-ut,  
then it is one in the following list:
\begin{enumerate}
\item[(1)] Suzuki groups $\mathrm{Sz}(q)$, potentially extended by field automorphisms ($n=q^2+1$);
\item[(2)] $\mathrm{AGL}(1,q)\le G\le \mathrm{A}\Gamma\mathrm{L}(1,q)$, where $q$ is either prime with $q \equiv 11 \mbox{ mod } 12$, or $q=2^p$ with $p$ prime, and  for all $c \in \mathrm{GF}(q)\setminus\{0,1\}$, 
$|\langle -1,c,c-1\rangle|=q-1$  $(n=q$);
\item[(3)]  subgroups of index $2$ in $\mathrm{AGL}(1,q)$, with $q \equiv 11 \mbox{ mod } 12$ and prime, and  for all $c \in \mathrm{GF}(q)\setminus\{0,1\}$, 
$|\langle -1,c,c-1\rangle|=q-1$  $(n=q$).
\end{enumerate}
\end{prop}
In the part~$(b)$ and $(c)$ above, we have denoted by $\mathrm{GF}(q)$ the finite field of cardinality $q$; moreover, for $c\in \mathrm{GF}(q)\setminus\{0,1\}$, we have denoted by $\langle -1,c,c-1\rangle$ the subgroup of the multiplicative group of $\mathrm{GF}(q)$ generated by $-1$, $c$ and $c-1$.
The $3$-ut property was confirmed by computation for all listed potential groups with $n\le 50$.

Recall that a \emph{regular two-graph}~\cite{taylor} on a set $\Omega$ is a set $T$ of $3$-subsets of $\Omega$ such that:
\begin{enumerate}\itemsep0pt
\item any $4$-subset contains an even number of members of $T$;
\item any two points lie in $\lambda$ members of $T$, for some fixed value $\lambda$.  
\end{enumerate}
It is not hard to verify that the groups in~\eqref{part2},~\eqref{part3},~\eqref{part4} and~%\eqref{part5} (with 
$(n,G)=(276,Co_3)$ of Proposition~\ref{p:3ut}  all have just two orbits on $3$-subsets and that  each orbit forms a regular two-graph, see for instance~\cite[Examples~$6.2$,~$6.3$ and~$6.4$ and Theorem~$6.7$]{taylor}. For these groups, we let $T$ and $T'$ denote these two orbits and we let $\lambda$ and $\lambda'$ denote the number of $3$-sets in $T$ and $T'$, respectively, that contain two given points (the valency as $3$-uniform hypergraphs). Moreover, we use the convention that $\lambda\le \lambda'$. 
%The values of $\lambda \le \lambda'$ for the groups of interest are:
%\begin{itemize}\itemsep0pt
%\item For $\Sp(2d,2)$ with $n=2^{2d-1}\pm2^{d-1}$, $k=2^{2d-2}$ or
%$k=2^{2d-2}\pm2^{d-1}-2$.
%\item For $2^{2d}:\Sp(2d,d)$ with $n=2^{2d}$, $k=2^{2d-1}$ or
%$2^{2d-1}-2$.
%\item For $\Co_3$, $k=112$ or $162$.
%\end{itemize}

\begin{prop}\label{p:3ut+}
Let $G$ be one of the groups in~$\eqref{part2}$,~$\eqref{part3}$,~$\eqref{part4}$  or~%\eqref{part5} (with 
$(n,G)=(276,Co_3)$ of Proposition~$\ref{p:3ut}$ and let $\lambda$ and $\lambda'$ be as above. Then $\lambda'+3\le t(G,3)$. Moreover, if $l>\lambda'+2$ and $l>\min\{3\lambda/2, (6\lambda'+9)/5\}$, then $G$ has the $(3,l)$-ut property.
\end{prop}

\begin{proof}
Let $T$ and $T'$ be the two $G$-orbits on $3$-subsets with parameters $\lambda \le \lambda'$, respectively.
 
Pick distinct points $x$ and $y$ and let $Z:=\{z\mid \{x,y,z\} \in T'\}$. By definition, $|Z|=\lambda'$. The partial $3$-partition $\{\{x\},\{y\}, Z\}$ witnesses that $G$ does not satisfy 
$(3,\lambda'+2)$-ut  and hence $t(G,3)\ge \lambda'+3$. 

Now consider an arbitrary partial partition $\{X,Y,Z\}$ with $|X| \le |Y|\le |Z|$.
We first treat the case $|X|=1$, that is, $X=\{x\}$ for some $x$. Consider the orbital graph of $G_x$ on $\Omega \setminus \{x\}$, such that $\{y,z\}$ is an edge if and only if $\{x,y,z\} \in T$. This graph is vertex-primitive, because in each of the groups under consideration $G_x$ acts primitively on $\Omega\setminus\{x\}$.  Hence its connectivity equals its valency 
$\lambda$ by~\cite{watkins}. If $1+|Y|+|Z|> n-\lambda$, the induced subgraph on $Y \cup Z$ is connected, and hence $\{X,Y,Z\}$ has a section from $T$.  By the same argument, it also contains a section of $T'$. It follows that if $G$ fails $(3,l)$-ut  with $l\ge n - \lambda +1=(\lambda+\lambda'+2)-\lambda+1=\lambda' +3$, this failure is not witnessed by a partition involving a singleton. 

Now consider the case $|X| \ge 2$, and assume that the partial partition $\{X,Y,Z\}$ witnesses the failure of $(3,l)$-ut for $G$, with $l \ge \lambda' +3$. Hence all sections of $\{X,Y,Z\}$ lie in either $T$ or $T'$.
Moreover, as in the proof for the $3$-ut property for the same groups in~\cite{ac}, it follows that for fixed $x_1,x_2 \in X$, all $3$-subsets of the form $\{x_1,x_2,w\}$ with $w \in Y \cup Z$  lie in the same $G$-orbit, and the corresponding result holds for any permutation of $X,Y,Z$. 

Define a graph $\Gamma$ on $X\cup Y \cup Z$ as follows. For distinct elements $v$ and $w$, pick an element $u$ in a part containing neither $v$ nor $w$, and let $v$ and $w$ be adjacent in $\Gamma$ if and
only if $\{u,v,w\} \in T'$. It follows from the previous paragraph that this definition does not depend on the choice of $u$. 

%As the partial partition $\{X,Y,Z\}$ witnesses the failure of $(3,l)$-ut for $G$, 
Note that pairs $\{u,v\}$ not contained in a part are either all edges or all non-edges of $\Gamma$. From this, it is easy to see that any $3$-subset of $X\cup Y\cup Z$
is contained in $T'$ if and only if it contains $1$ or $3$ edges.  

 If the subgraph induced by $\Gamma$ on $X$ is complete, then for  distinct $x,x' \in X$, $\{x,x',w\} \in T'$ for all 
$w \in (X \cup Y\cup Z)\setminus\{x,x'\}$. However, this contradicts $|X\cup Y\cup Z|=l \ge \lambda'+3$. Therefore, the subgraph induced by $\Gamma$ on $X$ is not complete and hence there exist two distinct points $x,x' \in X$ non-adjacent in $\Gamma$. Then $\{x,x',w\} \in T$ for  all $w \in Y \cup Z$, and 
hence $\lambda\ge |Y\cup Z|=l-|X|\ge l-l/3=2l/3$. Thus
$$l\le \frac{3\lambda}{2}.$$

Let $a=|X|$, let  $w \in Z$, let $\bar x \in X$ and let  $v_{\bar{x}}$ be the number of neighbours of $\bar{x}$ in $X$. Then $\bar{x}$ has $a-1-v_{\bar{x}}$ non-neighbours in $X$. In particular, there are $$\frac{(a-1)(a-2)}{2}-v_{\bar{x}}(a-v_{\bar{x}}-1)$$ pairs  $\{x,x'\}$  in $X$ such that either both or none of $x,x'$ is adjacent to $\bar x$. This number is minimal when $v_{\bar{x}}=(a-1)/2$ with minimum 
$$\frac{(a-1)(a-2)}{2} - \left(\frac{a-1}{2}\right)^2=\frac{(a-1)(a-3)}{4}.$$ As $\bar x$ was arbitrary, we deduce that there are at least $a(a-1)(a-3)/4$  pair-point combinations in $X$ with this property. Hence, for some  pair
$\{x,x'\}$ in $X$, there are at least  $(a-3)/2$ corresponding points. Suppose that $\{x,x'\}$ is an edge (a non-edge will result in an improved bound). Then 
$\{x,x',w\} \in T'$ for at least $(a-3)/2$ values $w \in X$, in addition to all $w\in Y \cup Z$. Hence
$$(a-3)/2 +|Y|+|Z| \le \lambda'.$$
We want to maximise $l=a+|Y|+|Z|$ with regard to the above constraint. This is achieved if $a$ is maximal, and thus for $a =  l/3=|Y|=|Z| $.
Hence  $(l/3-3)/2+l/3+l/3 \le \lambda'$ and so 
$ l\le (6\lambda'+9)/5.$
\end{proof}

\begin{proof}[Proof of Theorem~$\ref{k=3}$]
We need to consider only the groups arising from Proposition~\ref{p:3ut} part~\eqref{part2},~\eqref{part3},~\eqref{part4} and~\eqref{part5}, because the groups in~\eqref{part1} have $t(G,3)=3$ and any other group having the $3$-ut property has been reported in Table~$\ref{table:1}$ with no information on $t(G,3)$.

When $G$ is as in~\eqref{part2}, the proof follows from Proposition~\ref{p:3ut+} observing that $\lambda=\lambda'=(q-1)/2$, see~\cite{taylor}.
When $G$ is as in~\eqref{part3}, the proof follows from Proposition~\ref{p:3ut+} observing that $n = 2^{2d-1} \pm 2^{d-1}$ and $\{\lambda,\lambda'\}=\{2^{2d-2},2^{2d-2}\pm 2^{d-1}-2\}$, see~\cite{taylor} (the case $\Sp(6,2)$ with $n=28$ is the only one where $l>3\lambda/2$ is stronger than $l>(6\lambda'+9)/5$). When $G$ is as in~\eqref{part4}, the proof follows from Proposition~\ref{p:3ut+} observing that $n = 2^{2d}$, $\lambda=2^{2d-1}-2$ and $\lambda'=2^{2d-1}$, see~\cite{taylor}. Finally suppose that $G$ is as in~\eqref{part5}. If $(n,G)=(276,Co_3)$, then $\lambda=112$ and $\lambda'=162$ by~\cite{taylor} and hence the bounds in Table~\ref{table:1} follow from Proposition~\ref{p:3ut+}. The exact value for $t(G,3)$, when $(n,G)\in \{(7,\mathrm{AGL}(1,7)), (11,\mathrm{PSL}(2,11)),
(2^4:A_6)\}$, follows with a computation. The bound for $t(G,3)$, when $(n,G)\in \{(64,\mathrm{G}_2(2)), (64,\mathrm{G}_2(2)'), (65,\mathrm{Sz}(8)),(65,\mathrm{Aut}(\mathrm{Sz}(8))),(176,HS)\}$, follows again with the help of a computer: we have determined the smallest valency of a regular $G$-invariant $3$-uniform hypergraph for $G$  and we have applied Proposition~\ref{hyp}.
\end{proof}

\subsection{Case $k\ge 4$}
\begin{proof}[Proof of Theorem~$\ref{upperhalf}$]
This follows immediately from the remark preceding Theorem \ref{Living}.
\end{proof}
\begin{proof}[Proof of Theorem~$\ref{k>5}$]
When $k\le \lfloor (n+1)/2\rfloor$, the result follows from~\cite[Theorem~$1.4$]{ac}. When $\lfloor (n+1)/2 \rfloor < k \le n-6$, the result follows from~\cite[Theorem~$3.3$]{ac}.
\end{proof}

\begin{proof}[Proof of Theorem~$\ref{k=5}$]
From~\cite[Theorem~$1.5$]{ac}, we see that either $G$ is $5$-homogeneous (that is, $t(G,5)=5$), or $n=33$ and $G=\mathrm{P}\Gamma\mathrm{L}(2,32)$. In the second case, the determination of $t(G,5)$ follows with a computation.
\end{proof}
\begin{proof}[Proof of Theorem~$\ref{k=4}$]
From~\cite{abc-et}, we see that either $G$ is $4$-homogeneous (that is, $t(G,4)=4$), or $(n,G)$ is one of the four cases listed in the statement of Theorem~$\ref{k=4}$. For the first three cases, the value of $t(G,4)$ follows with a computation.
\end{proof}

\section{Regular semigroups}\label{semigroups}
 
 Arguably, three of the most important classes of semigroups are {\em groups}, {\em inverse} semigroups and {\em regular} semigroups, defined as follows. For a semigroup $S$, we have that
\begin{itemize}
\item  $S$ is a group if, for all $a\in S$, there exists a unique $b\in S$ such that $a=aba$;
\item  $S$ is  inverse if, for all $a\in S$, there exists a unique $b\in S$ such that $a=aba$ and $b=bab$;
\item  $S$ is regular if, for all $a\in S$, there exists $b\in S$ such that $a=aba$.  
\end{itemize}
Note that, if $a=aba$, then by setting $b':=bab$ we have $a=ab'a$ and $b'=b'ab'$.

To a large extent semigroup structure  theory  is about trying to show how the idempotents shape the structure of the semigroup, see for instance~\cite{ac}. Therefore it is no surprise that groups and inverse semigroups can be  characterized by their idempotents:
\begin{itemize}
\item a semigroup is inverse if and only it is regular and the idempotents commute (see \cite[Theorem 5.5.1]{Ho95});  
\item a semigroup is a group if and only if it is regular and contains exactly one idempotent (see \cite[Ex. 3.11]{Ho95}). 
 \end{itemize}   
 
 Inverse semigroups, apart from being the class of  semigroups with the largest number of books dedicated to them, were introduced by geometers and are still very important in the area \cite{Pa99}. 
 
Let $\Omega$ be a finite set.   The semigroup of partial transformations or functions on $\Omega$, denoted $\ptrans(\Omega)$, is regular, and all regular semigroups embed in some $\ptrans(\Omega)$; every group embeds in some $\ptrans(\Omega)$ as a group of permutations; and every inverse semigroup embeds in some $\ptrans(\Omega)$ as a semigroup of partial $1$-$1$  transformations. (This follows from the Vagner--Preston representation \cite[Theorem 5.1.7]{Ho95} that maps every inverse semigroup into an isomorphic semigroup of partial bijections on a set, and the analogous theorem of Cayley for groups.)
 
 The following result is stated for partial transformations but the proof is essentially in \cite{lmm};  we provide it just for the sake of completeness. For $t\in \ptrans( \Omega)$, let $\dom (t)$ denote the domain of $t$, $\im(t):=\Omega t$ and $\rank(t):=|\Omega t|$.    
 
   \begin{lemma}{{\rm (\cite[Lemma 2.3]{lmm})}}\label{mainlemma}
   Let $u\in \ptrans (\Omega)$ and let $G\le \ptrans (\Omega)$ be a group of permutations. Then the following are equivalent:
\begin{enumerate}
\item $u$ is regular in $\langle G,u\rangle$;
\item   there exists $g\in G$ such that $\rank( ugu)=\rank( u)$. 
\end{enumerate}
  \end{lemma}
\begin{proof}
Suppose $u$ is regular in  $\langle G,u\rangle$. Then, there exists $b\in  \langle G,u\rangle$ such that $u=ubu$. As $b=g_1u\cdots g_{m-1}ug_m$ for some $g_1,\ldots,g_m\in G$, it follows that 
$$\rank(u)=\rank(ubu)=\rank(ug_1u\cdots g_{m-1}ug_mu)\le \rank(ug_1u)\le \rank(u),$$
and hence $\rank(ug_1u)=\rank(u)$. 

Conversely, suppose $g\in G$ and $\rank(ugu)=\rank(u)$. Let $r:=\rank(u)$ and let $\{A_1,\ldots,A_r\}$ be the kernel of $u$.
Since $\rank(ugu)=\rank(u)$, $\im(ug)$ is a transversal for the kernel of $u$ and hence 
\begin{equation*}\rank((ug)^m)=\rank(u),
\end{equation*} for all natural numbers $m\ge 1$. Since $\Omega$ is a finite set, there exists a positive integer $\omega$ such that $(ug)^\omega$ is idempotent. As $\rank(u)=\rank((ug)^\omega)$, we deduce that $\{A_1,\ldots,A_r\}$ is also the kernel of $(ug)^\omega$. In particular, for each $i\in \{1,\ldots,r\}$, there exists $a_i\in A_i$ with $A_i(ug)^\omega=\{a_i\}$. 
Then $$A_i(ug)^\omega u=\{a_i\}u\subseteq A_iu,$$
for every $i\in \{1,\ldots,r\}$. This shows that $(ug)^\omega u=u$. If $\omega=1$, then $u=ugu$ and hence $u$ regular in $\langle G,u\rangle$. Similarly, if $\omega>1$, then $u=(ug)^\omega u=ug(ug)^{\omega-1}u=u\cdot g(ug)^{\omega-1}\cdot u$ and hence $u$ is regular in $\langle G,u\rangle$.
 \end{proof}

\begin{proof}[Proof of Theorem~$\ref{thrm:main}$]
It is clear that \ref{b)}) implies \ref{a)}). 

\smallskip

We now show that \ref{a)}) implies \ref{c)}). Let $A$ be a $k$-subset of $\Omega$, let $B$ an $l$-subset of $\Omega$ and let $P$ be a $k$-partition of $B$. Let $t\in \ptrans(\Omega)$ with $\dom(t)=B$ and with kernel  $P$. In particular, $\rank(t)=k$ and $|\dom(t)|=l$ and hence $t\in \ptrans_{k,l}(\Omega)$. As $\langle G,t\rangle$ is regular,  $t$ is regular in $\langle G,t\rangle$ and hence by Lemma~\ref{mainlemma} there exists $g\in G$ with $\rank(tgt)=\rank(t)$. Therefore $g$ maps $\im(t)=A$ into a transversal of the kernel $P$ of $t$. 

\smallskip

We finally prove  that \ref{c)}) implies \ref{b)}), when $k\le \max\{(n+1)/2,n-6\}$. Let $G$ possess the $(k,l)$-ut property, let $t\in \ptrans_{k,m}(\Omega)$ with $l\le m$ and let $v\in S:=\langle G,t\rangle$.  We let $r:=\rank(v)$, $B:=\dom(v)$, $A:=\im(v)$ and $P$ be the kernel of $v$. Observe that $r\le k$, and if $k=r$ then $|B| \ge l$. Suppose first that $G$ posseses the $(r,|B|)$-property. Thus, we deduce that there exists $g\in G$ such that $Ag=\im(vg)$ is a section of $P$ and hence $\rank(vgv)=\rank(v) $. Therefore, by Lemma~\ref{mainlemma}, we have that $v$ is regular in $\langle G,v\rangle$ and hence regular in $S$.
 
Suppose that $G$ does not possess the $(r,|B|)$-ut property. In particular, $G$ is not $r$-homogeneous by Proposition~\ref{up}~\eqref{warda} and~\eqref{wards}. Moreover, by Proposition~\ref{up}~\eqref{wards}, we have $1\le r<k$, as $k=r$ implies $|B| \ge l$. From Proposition~\ref{Living}   and from the remark preceding Proposition~\ref{Living}, we see that we have one of the following cases to deal with:
\begin{itemize}
\item $n=5$, $k=3$, $t(G,k)=5$ and $G$ is cyclic or dihedral,
\item $n=7$, $k=4$, $t(G,k)=7$ and $G$ is isomorphic to $\mathrm{AGL}(1,7)$.
\end{itemize}
Thus $l=m=n$, because $n\ge m\ge l\ge t(G,k)$. In particular, the domain of $t$ is the whole of $\Omega$. As $v\in S=\langle G,t\rangle$, we deduce that the domain of $v$ is also the whole of $\Omega$ and hence $|B|=n$. However, another direct inspection on these cases shows that $G$ posseses the $(r,n)$-ut property, contradicting the fact that we were assuming $G$ does not have the $(r,|B|)$-ut property. 
\end{proof}

  \section{Problems}

In the light of Theorems~\ref{k=4} and~\ref{k=3}, we propose the following problem.
\begin{problem} Improve upon the values of $t(G,3)$ in Table~\ref{table:1} and determine whether case~\eqref{case43} in Theorem~\ref{k=4} arises in general.
\end{problem}

A semigroup is said to be {\em idempotent generated} if it is generated by its own idempotents. 
\begin{problem}
Let $\Omega$ be a set of size $n$ and let $1\le k\le l < n$. Classify the permutation groups $G$ acting on $\Omega$ such that for all $t\in \ptrans_{k,l}(\Omega)$, the semigroup $\langle G,t\rangle\setminus G$ is idempotent generated.
\end{problem}

 A degree $n$ permutation group on $\Omega$ is said to have the $k$-et property if there exists some $k$-set $K\subseteq \Omega$ such that the orbit of $K$ under $G$ contains transversals for every $k$-partition of $\Omega$   \cite{abc-et}. The $k$-et property generalizes the $k$-ut property. Let  $k$ and $l$ be integers with $1\le k\le l\le n$. Then $G$ is said to have the \emph{$(k,l)$-existential transversal property}
(or $(k,l)$-et for short) if there exists a $k$-subset $A$ such that the orbit of $A$ contains a transversal for every $k$-partition of any $l$-subset $B$ of
$\Omega$.

\begin{problem}
Prove, for groups possessing the $(k,l)$-et property, results analogous to the ones in this paper on groups possessing the $(k,l)$-ut property. 
\end{problem} 

If the previous problem can be solved, then it will be possible to provide an analogue of Theorem~$\ref{thrm:main}$ for transformations with prescribed image (rather than just prescribed rank), a much stronger result.  

Let $V$ be a vector space of finite dimension $n$. For $1\le k\le l\le n$, let $\End_{k,l}(V)$ be the set of linear transformations $t:L\to K$, with $Lt=K$, where $L$ and $K$ are subspaces of $V$ of dimension $l$ and $k$, respectively.   

\begin{problem}
Classify the linear groups $G\le \aut (V)$ such that for all $t\in  \End_{k,l}(V)$ there exists $g\in G$ such that $\rank(tgt)=\rank(t)$. 
\end{problem}

The previous problem is certainly very difficult since easier versions of it (see \cite[Problems 10 and 11]{ac}) are still open.

\end{document}